\documentclass[]{article}
\usepackage{geometry}
\usepackage{amsthm}
\usepackage{mathrsfs} 
\usepackage{amsmath,amssymb}
\usepackage{indentfirst}
\newtheorem{theorem}{Theorem}[section]  
\newtheorem{definition}[theorem]{Definition}

\newtheorem{proposition}[theorem]{Proposition}
\newtheorem{lemma}[theorem]{Lemma}

\usepackage{color}
\usepackage{hyperref}
\hypersetup{
	colorlinks=true,
	linkcolor=blue,
	filecolor=magenta,
	urlcolor=cyan,
	citecolor=red
}
\begin{document}
\title{A density version of Waring-Goldbach problem}
\author{Meng Gao\\ \small{Department of Mathematics, China University of Mining and Technology,}\\ \small{Beijing 100083, P. R. China}\\ \small{E-mail: ntmgao@outlook.com}  }

\date{}
\maketitle

\begin{abstract}
Let $ \mathbb{P} $ denote the set of all primes and $ k \geq 2 $ be a positive integer. Suppose that $ A $ is a subset of $ \mathbb{P} $ with $ \underline{\delta}_{A}(\mathbb{P})>1-1/2k $, where $ \underline{\delta}_{A}(\mathbb{P}) $ is the lower density of $ A $ relative to $ \mathbb{P}$. We prove that every sufficiently large natural number $ n $ satisfying the necessary congruence condition can be written in the form $ n=p_{1}^{k}+\cdots+p_{s}^{k} $, where $ s > \max (16k\omega(k)+4k+3,\ k^{2}+k) $ and $ p_{i}\in A $ for all $ i \in \{ 1,\ldots,s\} $.	
	
\begin{flushleft}
\end{flushleft}

\end{abstract}

\section{Introduction}
For each prime $ p $, define $ \tau (k,p) $  so that $p^{\tau (k,p)}\ |\ k,\ p^{\tau (k,p)+1}\ \nmid\ k  $. Let

		\begin{center}
		$\begin{aligned} R_{k}:=\prod\limits_{(p-1)|k}p^{\gamma(k,p)}\end{aligned} $,
	\end{center}
where\begin{center}
  	$ \gamma(k,p):=  
  	\begin{cases}
		 \tau(k,p)+2 \ \quad if \   p=2 \  and \   \tau(k,p)>0   ,    \\
		 \tau(k,p)+1 \ \quad otherwise.
		\end {cases}$ 
	\end{center}
	
Waring-Goldbach problem is an important problem of additive number theory. The first result about Waring-Goldbach problem was established in \cite{Hua1}, in which Hua showed that all sufficiently large number $ n \equiv s(\bmod \ R_{k})$ can be represented as
\begin{center}
	$ n=p_{1}^{k}+\cdots+p_{s}^{k}, $
\end{center}
where $ p_{1},\ldots,p_{s} $ are primes and $ s>2^{k}. $ Since then the value of $ s $ has been constantly reduced. For the current records of the value of $ s $, one could refer to \cite{KW}. 

The transference principle was originally developed by Green \cite{Gre}, and now has become a powerful tool in additive number theory. Li, Pan \cite{LP} and Shao \cite{Shao} used the transference principle to study the density version of three primes theorem and obtained remarkable results. Similarly, Salmensuu \cite{Sal} used the transference principle to investigate the density version of Waring's problem. Motivated by Salmensuu's work \cite{Sal}, we study a density version of Waring-Goldbach problem.

 Let   $ A\subseteq \mathbb{P} $. Define
\begin{center}
	 $ \delta_{A}=\underline{\delta}_{A}(\mathbb{P}):=\liminf \limits_{N\rightarrow\infty}\dfrac{|A\bigcap[N]|}{|\mathbb{P}\bigcap[N]|} , $
\end{center}
where $ [N]:=\{1,\ldots,N\} $. And let $ \omega(n) $ denote the number of distinct prime divisors of $ n $.

Our main result is the following.

\begin{theorem}\label{thmcont:1.1}
	Let $ s,k \in \mathbb{N},\ k \geq 2,\ s > \max(16k\omega(k)+4k+3,\ k^{2}+k) $ and let $ \delta_{A}>1-1/2k $, then for every sufficiently large positive integer $  n \equiv s \ (\bmod \ R_{k}) $, we have $ n=p_{1}^{k}+\cdots+p_{s}^{k}, $  where $ p_{i} \in A $ for all $ i \in \{1,\ldots,s\} $.
\end{theorem}
Notice that as $ k $ approaches infinity, $ 1 - 1/2k $ approaches $ 1 $. This result is rather poor, so we aim to obtain a better result in terms of density. By relaxing the density condition, we have achieved a more favorable density threshold.

Let $ \mathcal{P}\subseteq  \mathbb{P} $ be such that \begin{center}
	$ \pi_{\mathcal{P}}(x) \sim \delta\pi(x) $,
\end{center}
where $ \delta \in (0,1) $ is a constant and $ \pi_{\mathcal{P}}(x):=\#\{p \leq x: p\in \mathcal{P}\} $. Then we have the following result.
\begin{theorem}\label{thmcont:1.2}
	Let $ s,k \in \mathbb{N},\ k \geq 2,\ s > \max(16k\omega(k)+4k+3,\ k^{2}+k) $ and let $ \delta>1/2 $, then for every sufficiently large positive integer $  n \equiv s \ (\bmod \ R_{k}) $, we have $ n=p_{1}^{k}+\cdots+p_{s}^{k}, $  where $ p_{i} \in \mathcal{P} $ for all $ i \in \{1,\ldots,s\} $.
\end{theorem}

\section{Notation}
Let $ s \in \mathbb{N} $ and $ s\geq 2 $. For a set $ A \in \mathbb{N} $, define the sumset by\begin{center}
	$ sA=\{a_{1}+\cdots+a_{s}:a_{1},\ldots,a_{s}\in A\} $
\end{center}
and define\begin{center}
	$ A^{(k)}=\{a^{k}:a\in A\} $.
\end{center}
For finitely supported functions $ f,g:\mathbb{Z}\rightarrow \mathbb{C} $ , we define convolution $ f\ast g  $ by 
\begin{center}
	$\begin{aligned}
	 f\ast g(n)=\sum\limits_{a+b=n}f(a)g(b) \end{aligned}$.
\end{center}

The Fourier transform of a finitely supported function $ f:\mathbb{Z}\rightarrow \mathbb{C} $ is defined by
\begin{center}
	$\begin{aligned} \widehat{f}(\alpha)=\sum\limits_{n\in \mathbb{Z}}f(n)e(n \alpha) \end{aligned}$ ,
\end{center}
where $ e(x)=e^{2\pi ix} $.
For $ x\in \mathbb{R} $ and $ q\in \mathbb{N} $,  we will also use notation $ e_{q}(x) $ as an abbreviation for $ e(x/q)$. 

For a set $ B $, we write $ 1_{B}(x) $ for  its characteristic function.  If $ f:B\rightarrow \mathbb{C} $ is a function and $ B_{1} $ is a non-empty finite subset of $ B $, we write $ \mathbb{E}_{x \in B_{1}}f(x) $ for the average value of $ f $ on  $ B_{1} $,  that is to say
\begin{center}
	$ \mathbb{E}_{x \in B_{1}}f(x)=\dfrac{1}{|B_{1}|}\sum\limits_{x \in B_{1}}f(x) $.
\end{center}
We write $ f=o(g) $ if 
\begin{center}
	$ \lim\limits_{x\rightarrow\infty} \dfrac{f(x)}{g(x)}=0$ .
\end{center}
The function $ f $ is asymptotic to $ g $, denoted $ f\sim g $ , if
\begin{center}
	$ \lim\limits_{x\rightarrow\infty} \dfrac{f(x)}{g(x)}=1$ .
\end{center}
We will use notation $ \mathbb{T} $ for $ \mathbb{R}/\mathbb{Z} $.
We also define the $ L^{p}$-$norm $ 
\begin{center}
	$\begin{aligned}
	 \|f\|_{p}=\bigg(\int_{\mathbb{T}}|f(\alpha)|^{p}d\alpha\bigg)^{1/p}\end{aligned} $
\end{center}
for a function $ f:\mathbb{T}\rightarrow \mathbb{C} $.

We write $ f \ll g $ or $ f=O(g) $ if there exists a constant $ C>0  $ such that $ |f(x)|\leq Cg(x) $ for all values of $ x $ in the domain of $ f $. If $ f\ll g $ and $ g \ll f $ we write $ f\asymp g $. The letter $ p $, with or without subscript,
denotes a prime number.

 For $ x \in \mathbb{R} $,
	\begin{center}
		$ ||x||=\min \{|x-z|:z\in \mathbb{Z}\} $.
	\end{center}

\section{Pseudorandom function and W-trick}
Let $ n_{0} $ be a sufficiently large positive integer satisfying $ n_{0}\equiv s \ (\bmod \ R_{k}) $ .
Let $ w=\log\log\log n_{0} $ and
\begin{equation}\label{equ1}
W:=\prod\limits_{1<p\leq w}p^{2k} .
\end{equation}
For $ m \in \mathbb{N} $, define $ \mathbb{Z}_{m}^{(k)}:=\{t^{k}:t \in \mathbb{Z}_{m}\} $ and $ Z(m):=\{a \in \mathbb{Z}_{m}^{(k)}:(a,m)=1 \} $, where $ \mathbb{Z}_{m}:=\mathbb{Z}/m\mathbb{Z} $.
Let $ b \in [W] $ be such that 
$ b \in Z(W) $. Define
\begin{equation}\label{equ2}
\sigma(b):=\#\{z \in [W]:z^{k}\equiv b(\bmod \ W)\}.
\end{equation}
Let $ N:=\lfloor 2n_{0}/(sW) \rfloor $. It is not difficult to prove that
\begin{center}
	$ W=o(\log N) $.
\end{center}

Let $ A, \mathcal{P} \subseteq \mathbb{P}  $ be defined as in Section 1. Define functions $ f_{b},\ \mathbf{f}_{b},\ \nu_{b}:[N]\rightarrow \mathbb{R}_{\geq 0} $ by
\begin{center}
	$  f_{b}(n):=\begin{cases}
	\dfrac{\varphi(W)}{W\sigma(b)}kp^{k-1}\log p \quad if \  Wn+b=p^{k}, \ p\in A,  \\
	0  \qquad\qquad\qquad\qquad \  otherwise, 
	\end {cases}  $
\end{center}

\begin{center}
	$  \mathbf{f}_{b}(n):=\begin{cases}
	\dfrac{\varphi(W)}{W\sigma(b)}kp^{k-1}\log p \quad if \  Wn+b=p^{k}, \ p\in \mathcal{P},  \\
	0  \qquad\qquad\qquad\qquad \  otherwise, 
	\end {cases}  $
\end{center}
and
\begin{center}
	$  \nu_{b}(n):=\begin{cases}
	\dfrac{\varphi(W)}{W\sigma(b)}kp^{k-1}\log p \quad if \  Wn+b=p^{k}, \ p\in \mathbb{P},  \\
	0  \qquad\qquad\qquad\qquad \ otherwise. 
	\end {cases}  $	
\end{center}

Define functions $ g,\ \mathbf{g}:[W]\times \mathbb{N}\rightarrow \mathbb{R}_{\geq 0} $ by
\begin{center}
$ g(b,N):=\mathbb{E}_{n \in [N]}f_{b}(n). $
\end{center}
and
\begin{center}
$ \mathbf{g}(b,N):=\mathbb{E}_{n \in [N]}\mathbf{f}_{b}(n). $
\end{center}
Obviously, $ f_{b}(n)\leq \nu_{b}(n) $ and $ \mathbf{f}_{b}(n)\leq \nu_{b}(n) $ for all $ n \in [N] $.
On the other hand, when $ b\in [W] $ with $ b\in Z(W) $, we have 
\begin{center}
	$ \begin{aligned}
	\mathbb{E}_{n \in [N]}\nu_{b}(n)
	&=\dfrac{1}{N}\sum\limits_{\substack{W+b\leq p^{k}\leq WN+b \\ p^{k}\equiv b (\bmod W) }}\dfrac{\varphi(W)}{W\sigma(b)}kp^{k-1}\log p\\
	&=\dfrac{1}{N}\sum\limits_{\substack{ p\leq (WN+b)^{1/k} \\ p^{k}\equiv b (\bmod W) }}\dfrac{\varphi(W)}{W\sigma(b)}kp^{k-1}\log p+O(W/N)\\
	&=\dfrac{1}{N\sigma(b)}\sum\limits_{\substack{z \in [W]\\ z^{k}\equiv b(\bmod W)}}\dfrac{\varphi(W)}{W}\sum\limits_{\substack{p\leq (WN+b)^{1/k} \\ p \equiv z (\bmod W)}}kp^{k-1}\log p+o(1).
	\end{aligned}$
\end{center}

Then, using the same method of Chow \cite[Section 2]{Chow}, we can prove that $ \mathbb{E}_{n \in [N]}\nu_{b}(n)\sim 1 $.
	We say that a function $ \nu :[N] \rightarrow \mathbb{R}_{\geq 0} $ is pseudorandom if $ \mid \widehat{\nu}(\alpha) - \widehat{1_{[N]}}(\alpha) \mid =o(N) $ for all $ \alpha \in \mathbb{T} $. In Section 5 we will prove that $ \nu_{b} $ is pseudorandom for all $ b \in [W] $ with $ b \in Z(W) $. The purpose of the W-trick in the definitions of $ f_{b} $ and $ \nu_{b} $ is to make pseudorandomness of $ \nu_{b} $ possible.

The notation of this section will be in force for the rest of the paper.

\section{Mean value estimate}

In this section, we will use Salmensuu's method to prove two mean results. 

\subsection{Local problem}
In \cite[Section 5]{Sal}, Salmensuu systematically studied Waring pair and proved a  local density version of Waring’s problem \cite[Proposition 5.2]{Sal}. Salmensuu's research findings related to Waring pairs are also essential components in proving our Theorem \ref{thmcont:1.1} and Theorem \ref{thmcont:1.2}.

 First, we introduce the definition of Waring pair.
\begin{definition}
Let $ q,s \in \mathbb{N}$. We say that $ (q,s) $ is a Waring pair if, for any $ B \subseteq Z(q)  $ with $ |B|>\dfrac{1}{2}|Z(q)| $	, we have $ sB=\{a\in \mathbb{Z}_{q}:a\equiv s\ (\bmod (R_{k},q)) \} $.
\end{definition} 
We will prove the following result.
\begin{proposition}\label{pro 4.2}
$ (W,s) $ is a Waring pair for any $ s\geq 8k\omega(k)+2k+2 $.
\end{proposition}

Please note that although our definition of $ W $ is different from that of Salmensuu \cite[Eq.(8)]{Sal}, we have obtained the same lower bound of $ s $ as Salmensuu.

The following lemma tells us how the elements in $ Z(p^{2k}) $ are distributed in certain cosets of $ p\cdot \mathbb{Z}_{p^{2k}} $.

\begin{lemma}\label{lem 4.3}
	Let $ p>2 $. For all $ a \in Z(p) $, we have \begin{center}
		$ |\{b\in \mathbb{Z}_{p^{2k}}^{(k)}: b\equiv a \ (\bmod\ p)\}|=p^{2k-1-\tau (k,p)} $.
	\end{center}
\end{lemma}
\begin{proof}
	See \cite[proof of Lemma 5.4]{Sal}.
\end{proof}
Next, we present a result on the local problem for prime power moduli.
\begin{lemma}\label{lem 4.4}
$ (p^{2k},s) $ is a Waring pair for all $ s\geq 8k $.	
\end{lemma}
\begin{proof}
	The case $ (p^{k},s) $ has been shown by Salmensuu \cite[Lemma 5.6]{Sal}. With minor changes the same proof also works for our case. The main difference is that when $ p-1\nmid k $ we use Lemma \ref{lem 4.3} in place of \cite[Lemma 5.4]{Sal}.
\end{proof}
$ \mathit{Proof \ of \ Proposition \ 4.2  } $. By \cite[Lemma 5.3]{Sal} and Lemma \ref{lem 4.4}, $ \big(\prod_{p\leq \omega,\ p|k}p^{2k},s\big) $ is a Waring pair for all $ s\geq 8k\omega(k) $. On the other hand, by \cite[Lemma 5.7]{Sal} and \cite[Lemma 5.9]{Sal}, $ \big(\prod_{p\leq \omega,\ p\nmid k}p^{2k},s\big) $ is a Waring pair for all $ s\geq 2k+2 $. Using \cite[Lemma 5.3]{Sal} and the previous results, $ (W,s) $ is a Waring pair for any $ s\geq 8k\omega(k)+2k+2 $.\qquad \qquad \qquad  \qquad \qquad \qquad  \qquad \qquad \qquad \qquad \qquad \qquad  \qquad \qquad  \qquad  \qquad\qquad \qquad \qquad\qquad \qquad \qquad \ \ \ \ \  $ \qedsymbol $

\subsection{Mean result}
First, we present a generalized version of the local problem.
\begin{lemma}\label{lem 4.5}
	Let $ f:Z(W)\rightarrow [0,1) $ satisfy $ \mathbb{E}_{b \in Z(W)}f(b)>1/2 $. Let $ s\geq 16k\omega(k)+4k+4 $. Then, for all $ n \in \mathbb{Z}_{W} $ with $  n \equiv s \ (\bmod \ R_{k}) $, there exist $ b_{1},\ldots,b_{s}\in Z(W) $ such that $ n\equiv b_{1}+\cdots+b_{s} (\bmod\ W),\ f(b_{i})>0 $ for all $ i \in \{1,\ldots,s\} $ and 
	\begin{center}
		$ f(b_{1})+\cdots+f(b_{s})>s/2. $
	\end{center}
\end{lemma}
\begin{proof}
	The case $ W=\prod_{1<p\leq w}p^{k} $ has been shown by Salmensuu \cite[Lemma 6.4]{Sal}. Using Proposition \ref{pro 4.2} instead of \cite[Proposition 5.2]{Sal}, the same proof also works for our case.
\end{proof}
As pointed out in \cite[Section 2]{Chow}, we have the following lemma.
\begin{lemma}\label{lem 4.6}
	For each $ b\in [W] $ with $ b \in Z(W),\ \sigma(b) $ does not depend on $ b $. In fact, we have
	\begin{center}
		$  \sigma(b)=\dfrac{\varphi(W)}{|Z(W)|} . $
	\end{center} 
\end{lemma}

Next, we prove two results concerning lower bounds for $ \mathbb{E}_{b \in Z(W)}g(b,N) $ and $ \mathbb{E}_{b \in Z(W)}\mathbf{g}(b,N) $.
\begin{lemma}\label{lem 4.7}
Let $ \epsilon \in (0,1/4) $. Then
\begin{center}
	$ \mathbb{E}_{b \in Z(W)}g(b,N) \geq k\delta_{A}-(k-1)-\epsilon $
\end{center}
provided that $ N $ is large enough depending on $ \epsilon $.
\end{lemma}
\begin{proof}
	In view of $ (\ref{equ1}) $, we have $ w<\log W $ and $ (p,W)=1 $ for any $ p >w $. Therefore, with $ p>(2W)^{1/k} $, we have $ (p,W)=1 $ provided that $ N $ is large enough.
	In addition, note that each prime $ p \in A  $ with $ (2W)^{1/k}<p\leq (WN+1)^{1/k} $ corresponds to a unique pair $ (n,b), \ n\in [N] , \ b \in [W]$ such that $ Wn+b=p^{k}, \ b\in Z(W)  $. Therefore, by Lemma \ref{lem 4.6}, we have
	\begin{center}
		$ \begin{aligned}
		\mathbb{E}_{b \in Z(W)}g(b,N)
		&=\dfrac{1}{|Z(W)|N}\sum\limits_{\substack{b \in [W] \\ b \in Z(W)}}\sum\limits_{\substack{n\in [N] \\ Wn+b=p^{k}\\ p \in A}}\dfrac{\varphi(W)}{W\sigma(b)}kp^{k-1}\log p\\
		&=\dfrac{1}{WN}\sum\limits_{\substack{b \in [W] \\ b \in Z(W)}}\sum\limits_{\substack{n\in [N] \\ Wn+b=p^{k}\\ p \in A}}kp^{k-1}\log p\\
		&\geq \dfrac{1}{WN}\sum\limits_{\substack{p \in A \\ (2W)^{1/k}<p\leq (WN+1)^{1/k} }}kp^{k-1}\log p. \\
		\end{aligned} $
	\end{center}
		Let $ X_{1}= (2W)^{1/k} $ and $ X_{2}= (WN+1)^{1/k}  $. Define
	\begin{center}
		$\begin{aligned} A(t):=\sum\limits_{\substack{p\leq t \\ p \in A}}1 \end{aligned}$
	\end{center}
	and
	\begin{center}
		$ g(t):=kt^{k-1}\log t $.
	\end{center}
Next, we give a lower bound for \begin{center}
	$\begin{aligned} \sum\limits_{\substack{ p \in A \\ X_{1}<p\leq X_{2}}}kp^{k-1}\log p .\end{aligned}$
\end{center} 
	This sum is treated using partical summation.	
\begin{center}
	$ \begin{aligned}
	\sum\limits_{\substack{ p \in A \\ X_{1}<p\leq X_{2}}}kp^{k-1}\log p 
	&=A(X_{2})g(X_{2})-A(X_{1})g(X_{1})-\int_{X_{1}}^{X_{2}}A(t)g'(t)dt\\
	&\geq (1-o(1))\delta_{A}(X_{2}/\log X_{2})g(X_{2})-(1+o(1))(X_{1}/\log X_{1})g(X_{1})\\
	&\ \ \ -(1+o(1))\int_{X_{1}}^{X_{2}}\big(t/\log t\big)\big(k(k-1)t^{k-2}\log t+kt^{k-2}\big)dt\\
	&=(1-o(1))\delta_{A}kX_{2}^{k}-(1+o(1))(k-1)\int_{X_{1}}^{X_{2}}kt^{k-1}dt-(1+o(1))\int_{X_{1}}^{X_{2}}(kt^{k-1}/\log t) dt\\
	&\geq (1-o(1))\delta_{A}kX_{2}^{k}-(1+o(1))(k-1)X_{2}^{k}-(1+o(1))X_{2}^{k}/\log X_{1}\\
	&=\big(k\delta_{A}-(k-1)-o(1) \big)X_{2}^{k}.
	 \end{aligned}$
\end{center}
Therefore, we have 
\begin{center}
	$ \mathbb{E}_{b \in Z(W)}g(b,N)\geq k\delta_{A}-(k-1)-\epsilon $
\end{center}
provided that $ N $ is large enough depending on $ \epsilon $.	
\end{proof}
\begin{lemma}\label{lem 4.8}
Let $ \epsilon \in (0,1/6) $. Then
\begin{center}
	$ \mathbb{E}_{b \in Z(W)}\mathbf{g}(b,N) \geq (1-\epsilon)\delta $
\end{center}
provided that $ N $ is large enough depending on $ \epsilon $.	
\end{lemma}
\begin{proof}
Similarly, we have
\begin{center}
	$\begin{aligned} \mathbb{E}_{b \in Z(W)}\mathbf{g}(b,N) \geq \dfrac{1}{WN}\sum\limits_{\substack{p \in \mathcal{P} \\ (2W)^{1/k}<p\leq (WN+1)^{1/k} }}kp^{k-1}\log p. \end{aligned}$
\end{center}
Again using partical summation, we have
\begin{center}
	$\begin{aligned}
 \sum\limits_{\substack{p \in \mathcal{P} \\ X_{1}<p\leq X_{2} }}kp^{k-1}\log p 
&=\pi_{\mathcal{P}}(X_{2})g(X_{2})-\pi_{\mathcal{P}}(X_{1})g(X_{1})-\int_{X_{1}}^{X_{2}}\pi_{\mathcal{P}}(t)g'(t)dt	\\
&\geq (1-o(1))\delta(X_{2}/\log X_{2})g(X_{2})-(1+o(1))\delta(X_{1}/\log X_{1})g(X_{1})\\
&\ \ \ -(1+o(1))\delta\int_{X_{1}}^{X_{2}}\big(t/\log t\big)\big(k(k-1)t^{k-2}\log t+kt^{k-2}\big)dt\\
&=(1-o(1))k\delta X_{2}^{k}-(1+o(1))(k-1)\delta \int_{X_{1}}^{X_{2}}kt^{k-1}dt-(1+o(1))\delta \int_{X_{1}}^{X_{2}}(kt^{k-1}/\log t)dt\\
&\geq (1-o(1))k\delta X_{2}^{k}-(1+o(1))(k-1)\delta X_{2}^{k}-(1+o(1))\delta X_{2}^{k}/\log X_{1}\\
&=(1-o(1))\delta X_{2}^{k}.
\end{aligned}$
\end{center}
Therefore, we have 
\begin{center}
	$ \mathbb{E}_{b \in Z(W)}\mathbf{g}(b,N)\geq (1-\epsilon)\delta $
\end{center}
provided that $ N $ is large enough depending on $ \epsilon $.
\end{proof}
Finally, we use these two results and Lemma \ref{lem 4.5} to prove two mean results.
\begin{proposition}\label{pro 4.9}
Let $ \epsilon \in (0,1/4) $ and let $ N $ be sufficiently large depending on $ \epsilon $. Let $ \delta_{A}>1-1/2k+2\epsilon/k $  and $ s\geq 16k\omega(k)+4k+4 $. Then, for all $ n \in \mathbb{Z}_{W} $ with $ n \equiv s \ (\bmod \ R_{k}) $, there exist $ b_{1},\ldots,b_{s} \in Z(W) $ such that $ n \equiv b_{1}+     \cdots+b_{s}(\bmod \ W ), \ g(b_{i},N) > \epsilon/2$ for all $ i \in \left\lbrace 1,\ldots,s\right\rbrace  $ and 
\begin{center}
	$ g(b_{1},N)+\cdots+g(b_{s},N) > s(1+\epsilon)/2 $.
\end{center}	
\end{proposition}
\begin{proof}
	Note that $ \delta_{A}>1-1/2k+2\epsilon/k $ is equivalent to $ k\delta_{A}-(k-1)-\epsilon>1/2+\epsilon $. Therefore, by Lemma \ref{lem 4.7}, we have \begin{center}
		$ \mathbb{E}_{b \in Z(W)}g(b,N) > 1/2+\epsilon $
	\end{center}
provided that $ N $ is large enough depending on $ \epsilon $.
For $ b \in Z(W) $, define
\begin{center}
$ f(b):=\max\bigg(0,\ \dfrac{1}{1+\epsilon}\big(g(b,N)-\epsilon/2\big)\bigg) $ . 
\end{center}
Recall that $ \mathbb{E}_{n \in [N]}\nu_{b}(n)\sim 1 $. We have $ g(b,N)\leq 1+\epsilon $ provided that $ N $ is large enough depending on $ \epsilon $. Therefore, $ f(b)  \in [0,1)$ . In addition,
\begin{center}
	$ \mathbb{E}_{b \in Z(W)}f(b)\geq \dfrac{1}{1+\epsilon}\mathbb{E}_{b \in Z(W)}\big(g(b,N)-\epsilon/2\big)>1/2 $.
\end{center}
Hence, by Lemma \ref{lem 4.5}, for all $ n \in \mathbb{Z}_{W} $ with $ n \equiv s \ (\bmod\ R_{k}) $, there exist $ b_{1},\ldots,b_{s}\in Z(W) $ such that $ n\equiv b_{1}+\cdots+b_{s} \ (\bmod \ W), \ f(b_{i})>0  $ for all $ i\in \{1,\ldots,s\} $ and
\begin{center}
	$ f(b_{1})+\cdots+f(b_{s})>s/2 . $
\end{center}
By definition of $ f $, we get that $ g(b_{i},N)>\epsilon/2 $ for all $ i \in \{1,\ldots,s\} $ and 
\begin{center}
	$ g(b_{1},N)+\cdots+g(b_{s},N)>\dfrac{s(1+\epsilon)}{2}+\dfrac{s\epsilon}{2}>\dfrac{s(1+\epsilon)}{2} $.
\end{center}
\end{proof}

\begin{proposition}\label{pro 4.10}
	Let $ \epsilon \in (0,1/6) $ and let $ N $ be sufficiently large depending on $ \epsilon $. Let $ \delta>1/2+3\epsilon $  and $ s\geq 16k\omega(k)+4k+4 $. Then, for all $ n \in \mathbb{Z}_{W} $ with $ n \equiv s \ (\bmod \ R_{k}) $, there exist $ b_{1},\ldots,b_{s} \in Z(W) $ such that $ n \equiv b_{1}+\cdots+b_{s}(\bmod \ W ), \ \mathbf{g}(b_{i},N) > \epsilon/2$ for all $ i \in \left\lbrace 1,\ldots,s\right\rbrace  $ and 
	\begin{center}
		$ \mathbf{g}(b_{1},N)+\cdots+\mathbf{g}(b_{s},N) > s(1+\epsilon)/2 $.
	\end{center}	
\end{proposition}
\begin{proof}
	Note that $ \delta>1/2+3\epsilon $ .
	Therefore, by Lemma \ref{lem 4.8}, we have \begin{center}
		$ \mathbb{E}_{b \in Z(W)}\mathbf{g}(b,N) > (1-\epsilon)( 1/2+3\epsilon)>1/2+2\epsilon $
	\end{center}
	provided that $ N $ is large enough depending on $ \epsilon $.
	For $ b \in Z(W) $, define
	\begin{center}
		$ \mathbf{f}(b):=\max\bigg(0,\ \dfrac{1}{1+\epsilon}\big(\mathbf{g}(b,N)-\epsilon/2\big)\bigg) $ . 
	\end{center}
	Recall that $ \mathbb{E}_{n \in [N]}\nu_{b}(n)\sim 1 $. We have $ \mathbf{g}(b,N)\leq 1+\epsilon $ provided that $ N $ is large enough depending on $ \epsilon $. Therefore, $ \mathbf{f}(b)  \in [0,1)$ . In addition,
	\begin{center}
		$ \mathbb{E}_{b \in Z(W)}\mathbf{f}(b)\geq \dfrac{1}{1+\epsilon}\mathbb{E}_{b \in Z(W)}\big(\mathbf{g}(b,N)-\epsilon/2\big)>1/2 $.
	\end{center}
	Hence, by Lemma \ref{lem 4.5}, for all $ n \in \mathbb{Z}_{W} $ with $ n \equiv s \ (\bmod\ R_{k}) $, there exist $ b_{1},\ldots,b_{s}\in Z(W) $ such that $ n\equiv b_{1}+\cdots+b_{s} \ (\bmod \ W), \ \mathbf{f}(b_{i})>0  $ for all $ i\in \{1,\ldots,s\} $ and
	\begin{center}
		$ \mathbf{f}(b_{1})+\cdots+\mathbf{f}(b_{s})>s/2 . $
	\end{center}
	By definition of $ \mathbf{f} $, we get that $ \mathbf{g}(b_{i},N)>\epsilon/2 $ for all $ i \in \{1,\ldots,s\} $ and 
	\begin{center}
		$ \mathbf{g}(b_{1},N)+\cdots+\mathbf{g}(b_{s},N)>\dfrac{s(1+\epsilon)}{2}+\dfrac{s\epsilon}{2}>\dfrac{s(1+\epsilon)}{2} $.
	\end{center}
\end{proof}

\section{Pseudorandomness}
In this section we will use the standard circle method machinery to prove that $ \nu_{b} $ is pseudorandom. In other words, we will prove the following result.
\begin{proposition}\label{pro 5.1}
	Let $ \alpha \in \mathbb{T} $ . For $ b\in [W] $ with $ b \in Z(W) $, we have
	\begin{center}
		$ \mid \widehat{\nu}_{b}(\alpha) - \widehat{1_{[N]}}(\alpha) \mid =o(N). $
	\end{center}
\end{proposition}
Our proof of Proposition \ref{pro 5.1} will follow that of Chow in \cite[Section 3]{Chow} and \cite[Section 4]{Chow}, only with some slight modifications. First, we introduce the Hardy and Littlewood decomposition. Let $ \sigma_{0} $ be a large positive constant, and let $ \sigma $ be a much larger positive constant. Let $ L=\log(WN+W)^{1/k} , \ P=L^{\sigma},\ Q=WN/L^{\sigma}$. For $ q \in \mathbb{N} $ and $ a \in \mathbb{Z} $, write $ \mathfrak{M}(q,a):=\{\alpha \in \mathbb{T}:|\alpha-a/q|\leq 1/Q\} $. Let 
\begin{center}
	$\begin{aligned}  \mathfrak{M}(q):=\bigcup\limits^{q-1}_{\substack{a=0  \\ (a,q)=1}}\mathfrak{M}(q,a)  \end{aligned}$
\end{center} 
and
\begin{center}
	$\begin{aligned}  \mathfrak{M}:=\bigcup\limits^{q-1}_{\substack{a=0  \\ (a,q)=1 \\ 1\leq q \leq P}}\mathfrak{M}(q,a)=\bigcup\limits_{1\leq q \leq P}\mathfrak{M}(q) \end{aligned} $.
\end{center} 
Put $ \mathfrak{m}=\mathbb{T}\setminus\mathfrak{M} $. We call $ \mathfrak{M} $ major arcs and $ \mathfrak{m} $ minor arcs.
For any $ b \in [W] $ with $ b \in Z(W) $, we have 
\begin{equation}\label{equ3}
\begin{aligned}
\widehat{\nu_{b}}(\alpha)
&=\dfrac{\varphi(W)}{W\sigma(b)}\sum\limits_{\substack{ W+b \leq p^{k}\leq  WN+b  \\ p^{k}\equiv b (\bmod W)}}kp^{k-1}(\log p)e\bigg(\alpha \dfrac{p^{k}-b}{W}\bigg)\\
&=\dfrac{\varphi(W)}{W\sigma(b)}\sum\limits_{\substack{ p^{k}\leq  WN+b  \\ p^{k}\equiv b (\bmod W)}}kp^{k-1}(\log p)e\bigg(\alpha \dfrac{p^{k}-b}{W}\bigg)+O(W)\\
&=\dfrac{\varphi(W)e(-\alpha b/W)}{W\sigma(b)}\sum\limits_{\substack{z \in [W] \\ z^{k}\equiv b (\bmod W)}}\sum\limits_{\substack{p\leq Y \\ p \equiv z (\bmod W)}}(kp^{k-1}\log p)e(\alpha p^{k}/W)+O(\log N),
\end{aligned}
\end{equation}
where $ Y=\lfloor (WN+b)^{1/k} \rfloor $. Therefore, we focus on the inner sum
\begin{center}
	$\begin{aligned} \sum\limits_{\substack{p\leq Y \\ p \equiv z (\bmod W)}}(kp^{k-1}\log p)e(\alpha p^{k}/W). \end{aligned}$
\end{center}
\subsection{Minor arcs}
In this subsection, we use \cite[Theorem 10]{Hua2} to prove Proposition \ref{pro 5.1} when $ \alpha \in \mathfrak{m} $.
\begin{lemma}\label{lem 5.2}
	If $ \alpha \in \mathfrak{m} $ then $ \widehat{\nu_{b}}(\alpha)\ll NL^{-\sigma_{0}} $.	
\end{lemma}
\begin{proof}
	If $ \alpha \in \mathfrak{m} $, by Dirichlet's approximation theorem \cite[Lemma 2.1]{Vau}, there exists a rational number $ a/q $ with $ (a,q)=1,\ 1\leq q \leq Q $ and $ |\alpha-a/q|\leq 1/(qQ) $. Because of $ \alpha \notin \mathfrak{M} $, we have $ q>L^{\sigma} $. Hence, let $ \beta=\alpha-a/q $, we have 
	\begin{center}
		$ |\beta|\leq \dfrac{1}{qQ}\leq \dfrac{1}{WN}. $
	\end{center}	
	Let 
	\begin{center}
		$\begin{aligned} A_{\diamond}(t)=\sum\limits_{\substack{p\leq t \\ p \equiv z (\bmod W)}}e_{Wq}(ap^{k}) \end{aligned} $
	\end{center}
	and 
	\begin{equation}\label{equ4}
	f(t)=e(\beta t^{k}/W)kt^{k-1}\log t.
	\end{equation}
	Trivially $ A_{\diamond}(t)\ll t $ . Note that $ f'(t)\ll Y^{k-2}\log Y  $. Using partial summation, we have 
	\begin{center}$
		\begin{aligned}
		\sum\limits_{\substack{p\leq Y \\ p \equiv z (\bmod W)}}(kp^{k-1}\log p)e(\alpha p^{k}/W)
		&= \sum\limits_{\substack{p\leq Y \\ p \equiv z (\bmod W)}}e(\beta p^{k}/W)(kp^{k-1}\log p)e_{Wq} (ap^{k})\\
		&=A_{\diamond}(Y)f(Y)-\int_{1}^{Y}A_{\diamond}(t)f'(t)dt\\
		&=A_{\diamond}(Y)f(Y)-\int_{YL^{-2\sigma_{0}}}^{Y}A_{\diamond}(t)f'(t)dt+O(NL^{-\sigma_{0}}).
		\end{aligned}$
	\end{center}
	In view of  (\ref{equ2})   and   (\ref{equ3})  , we have
	\begin{equation}\label{equ5}
	\widehat{\nu_{b}}(\alpha)\ll NL^{-\sigma_{0}}+Y^{k-1}L\sup_{YL^{-2\sigma_{0}}<t\leq Y}|A_{\diamond}(t)|.
	\end{equation}
	Next, we use \cite[Theorem 10]{Hua2} to estimate $ A_{\diamond}(t) $ when $ YL^{-2\sigma_{0}}<t\leq Y $. Clearly $ 0<W\leq \log (YL^{-2\sigma_{0}})$ provided that $ N $ is large enough . Let $  a'=a/(a,W) $ and $ W'=W/(a,W) $, we can get
	\begin{center}
		$ L^{\sigma}<q\leq W'q \leq L\dfrac{WN}{L^{\sigma}}=WNL^{-\sigma+1}\leq (YL^{-2\sigma_{0}})^{k}L^{-\sigma/2} $
	\end{center}
	provided that $ N $ is large enough. By \cite[Theorem 10]{Hua2}, with $ YL^{-2\sigma_{0}}<t\leq Y $ and $ L\asymp \log t $, we have 
	\begin{equation}\label{equ6}
	A_{\diamond}(t)\ll YL^{-\sigma_{0}-1}W^{-1}.
	\end{equation}
	Therefore, by $ (\ref{equ5}) $ and $ (\ref{equ6}) $, we have
	\begin{center}
		$ \widehat{\nu_{b}}(\alpha)\ll NL^{-\sigma_{0}}+Y^{k}L^{-\sigma_{0}}W^{-1}\ll NL^{-\sigma_{0}} $.
	\end{center}	
\end{proof}
\begin{lemma}\label{lem 5.3}
	Let $ \alpha \in \mathfrak{m} $. Then 
	\begin{center}
		$\widehat{\nu}_{b}(\alpha)-\widehat{1_{[N]}}(\alpha)\ll NL^{-\sigma_{0}} $.
	\end{center}
\end{lemma}
\begin{proof}
	Again using Dirichlet's approximation theorem \cite[Lemma 2.1]{Vau}, we have that there exists a rational number $ a/q $ with $ (a,q)=1,\ 1 \leq q \leq L^{\sigma} $ and $ |\alpha-a/q|\leq \dfrac{1}{qL^{\sigma}} $. As $ \alpha \notin \mathfrak{M} $, we must have $ |\alpha-a/q|> \dfrac{L^{\sigma}}{WN} $. Hence,
	\begin{center}
		$ \widehat{1_{[N]}}(\alpha)\ll ||\alpha||^{-1}\leq \dfrac{q}{||q\alpha||}\leq \dfrac{WN}{L^{\sigma}}\ll NL^{-\sigma+1} $.
	\end{center}
	Therefore, by Lemma \ref{lem 5.2}, we have
	\begin{center}
		$ \widehat{\nu}_{b}(\alpha)-\widehat{1_{[N]}}(\alpha)\ll NL^{-\sigma_{0}}+NL^{-\sigma+1}\ll NL^{-\sigma_{0}} $.
	\end{center}
\end{proof}
\subsection{Major arcs}
When $ (z, W)=1 $, let 
\begin{equation}\label{equ7}
S_{q}^{\ast}(a,z)=\sum\limits_{\substack{r=0 \\ (z+Wr,Wq)=1}}^{q-1}e_{q}\bigg(a\dfrac{(z+Wr)^{k}-b}{W}\bigg)
\end{equation}
and 
\begin{center}
	$ \begin{aligned}
	I(\beta)=\int_{0}^{N}e(\beta t)dt.
	\end{aligned}  $
\end{center}

First, we prove the following lemma.
\begin{lemma}\label{lem 5.4}
	Let $ \alpha \in \mathfrak{M}(q,a) $ with $ (a,q)=1 $	and $ q\leq L^{\sigma} $, and put
	\begin{center}
		$ \beta =\alpha-a/q \in [-L^{\sigma}/(WN),L^{\sigma}/(WN)]. $
	\end{center}
	Then
	\begin{center}
	$\widehat{\nu_{b}}(\alpha)=\dfrac{\varphi(W)}{\varphi(Wq)\sigma(b)}\sum\limits_{\substack{z \in [W] \\ z^{k}\equiv b (\bmod W)}}S_{q}^{\ast}(a,z)I(\beta)+O(Ne^{-C_{3}\sqrt{L}}) $,
	\end{center}
	where $ C_{3} $	is a positive constant depending only on $ \sigma $.
\end{lemma}
\begin{proof}
	Fix $ b \in [W] $ with $  b \in Z(W) $.
	For $ n \in [Y] $, let 
	\begin{center}
		$\begin{aligned}
		S_{n}=\sum\limits_{\substack{p\leq n \\ p \equiv z (\bmod W)}}e_{Wq}(ap^{k}).\end{aligned} $
	\end{center}
	In addition,
	\begin{center}
		$ \begin{aligned}
		S_{n}&=\sum\limits_{\substack{p\leq n \\ p \equiv z (\bmod W) \\ (p,Wq)=1}}e_{Wq}(ap^{k})+O(Wq)\\
		&=\sum\limits_{\substack{r=0 \\ (z+Wr,Wq)=1}}^{q-1}e_{Wq}(a(z+Wr)^{k})\sum\limits_{\substack{p\leq n\\ p \equiv z+Wr (\bmod Wq)}}1+O(Wq).
		\end{aligned} $	
	\end{center}
	Let
	\begin{center}
		$\begin{aligned} V_{q}(a,z)=\sum\limits_{\substack{r=0 \\ (z+Wr,Wq)=1}}^{q-1}e_{Wq}(a(z+Wr)^{k}).\end{aligned} $
	\end{center}
	In view of $ Wq \leq L^{\sigma+1},\ (z+Wr,Wq)=1 $ and $ n \leq Y \leq (WN+W)^{1/k} $, by Siegel-Walfisz theorem \cite[Lemma 7.14]{Hua2}  , we have
	\begin{center}
		$\begin{aligned} \sum\limits_{\substack{p\leq n\\ p \equiv z+Wr (\bmod Wq)}}1 =\dfrac{\mathtt{Li}(n)}{\varphi(Wq)}+O\big((WN)^{1/k}e^{-C_{1}\sqrt{L}}\big),\end{aligned}$
	\end{center}
	where  $ C_{1} $ is a positive constant depending only on $ \sigma $ .
	Hence, we have 
	\begin{equation}\label{equ8}
	S_{n}=\dfrac{\mathtt{Li}(n)}{\varphi(Wq)}V_{q}(a,z)+O\big((WN)^{1/k}e^{-C_{2}\sqrt{L}}\big),
	\end{equation}
	where  $ C_{2} $ is a positive constant depending only on $ \sigma $ .
	Using Abel summation, we have 
	\begin{equation}\label{equ9}
	\begin{aligned}  \sum\limits_{\substack{p\leq Y \\ p \equiv z (\bmod W)}}(kp^{k-1}\log p)e(\alpha p^{k}/W)
	&=\sum\limits_{n=2}^{Y}(S_{n}-S_{n-1})f(n)\\
	&=S_{Y}f(Y+1)+\sum\limits_{n=2}^{Y}S_{n}(f(n)-f(n+1)) ,
	\end{aligned}  
	\end{equation}
	where $ f(t) $ has been defined in  (\ref{equ4}) .
	
	For $ n\leq Y $ and $ |\beta|\leq \dfrac{L^{\sigma}}{WN} $, we have
	\begin{equation}\label{equ10}
	f(n)-f(n+1)\ll Y^{k-2}L^{\sigma+1}.
	\end{equation}
	By (\ref{equ8}),\ (\ref{equ9}) and (\ref{equ10}), we have
	\begin{center}
		$ \begin{aligned}
		\sum\limits_{\substack{p\leq Y \\ p \equiv z (\bmod W)}}(kp^{k-1}\log p)e(\alpha p^{k}/W)&=\dfrac{V_{q}(a,z)}{\varphi(Wq)}\bigg[\mathtt{Li}(Y)f(Y+1)+\sum\limits_{n=2}^{Y}\mathtt{Li}(n)(f(n)\\
		&-f(n+1))  \bigg]
		+O(Ne^{-C_{3}\sqrt{L}}) ,
		\end{aligned}$
	\end{center}
	where  $ C_{3} $ is a positive constant depending only on $ \sigma $ .
	In view of $ \mathtt{Li}(2)=0 $, we have
	\begin{center}
		$\begin{aligned}
		\sum\limits_{\substack{p\leq Y \\ p \equiv z (\bmod W)}}(kp^{k-1}\log p)e(\alpha p^{k}/W)=\dfrac{V_{q}(a,z)}{\varphi(Wq)}\sum\limits_{n=3}^{Y}\int_{n-1}^{n}\dfrac{f(n)}{\log t}dt+O(Ne^{-C_{3}\sqrt{L}}) \end{aligned}$.
	\end{center}
	On the other hand,
	\begin{center}
		$ f(n)=f(t)+O(Y^{k-2}L^{\sigma+1}) $
	\end{center}
	for $ 2\leq n-1<t<n $.
	Therefore,
	\begin{center}
		$\begin{aligned}
		\sum\limits_{n=3}^{Y}\int_{n-1}^{n}\dfrac{f(n)}{\log t}dt&=\int_{2}^{Y}kt^{k-1}e(\beta t^{k}/W)dt+O(Y^{k-1}L^{\sigma+1})\\
		&=W\int_{2^{k}/W}^{Y^{k}/W}e(\beta t)dt+O(Y^{k-1}L^{\sigma+1}) \\
		&=WI(\beta)+O(Y^{k-1}L^{\sigma+1}).
		\end{aligned}$
	\end{center}
	It follows that 
	\begin{equation}\label{equ11}
	\sum\limits_{\substack{p\leq Y \\ p \equiv z (\bmod W)}}(kp^{k-1}\log p)e(\alpha p^{k}/W)=\dfrac{W}{\varphi(Wq)}V_{q}(a,z)I(\beta)+O(Ne^{-C_{3}\sqrt{L}}) .
	\end{equation}
	In view of (\ref{equ3}) and (\ref{equ11}), we have
	\begin{center}
		$\begin{aligned} \widehat{\nu_{b}}(\alpha)=\dfrac{\varphi(W)e(-\alpha b/W)}{\varphi(Wq)\sigma(b)}\sum\limits_{\substack{z \in [W]\\ z^{k}\equiv b (\bmod W)}}V_{q}(a,z)I(\beta)+O(Ne^{-C_{3}\sqrt{L}}). \end{aligned}$
	\end{center}
	Obviously,
	\begin{center}
		$ e_{Wq}(-ab)V_{q}(a,z)=S_{q}^{\ast}(a,z) $.
	\end{center}
	Therefore,
	\begin{center}
		$\begin{aligned} \widehat{\nu_{b}}(\alpha)=\dfrac{\varphi(W)e(-\beta b/W)}{\varphi(Wq)\sigma(b)}\sum\limits_{\substack{z \in [W]\\ z^{k}\equiv b (\bmod W)}}S_{q}^{\ast}(a,z)I(\beta)+O(Ne^{-C_{3}\sqrt{L}}) .\end{aligned}$
	\end{center}
	Since 
	\begin{center}
		$ \begin{aligned}
		e(-\beta b/W)I(\beta)=\int_{0}^{N}e(\beta(t-b/W))dt=I(\beta)+O(1), \end{aligned}$
	\end{center}
	Lemma \ref{lem 5.4} follows readily.
\end{proof}

Since we have established the approximation lemma for $ \widehat{\nu_{b}} $ when $ \alpha \in \mathfrak{M} $, now we are ready to finish the proof of Proposition \ref{pro 5.1} by tackling the major arcs case.

Next, we consider the case $ q=1 $ .
\begin{lemma}\label{lem 5.5}
	Let $ \alpha \in \mathfrak{M}(1)  $ . Then 
	\begin{center}
		$\widehat{\nu}_{b}(\alpha)-\widehat{1_{[N]}}(\alpha)\ll Ne^{-C_{3}\sqrt{L}} $.
	\end{center}
\end{lemma}	
\begin{proof}
	By (\ref{equ7}) and Lemma \ref{lem 5.4},\ we have 
	\begin{center}
		$ \widehat{\nu_{b}}(\alpha)=I(\alpha)+O(Ne^{-C_{3}\sqrt{L}}) $.
	\end{center}	
	Using the Euler-Maclaurin summation formula \cite[Eq.(4.8)]{Vau}, we have
	\begin{center}
		$\begin{aligned}
		\widehat{1_{[N]}}(\alpha)=\sum\limits_{n=1}^{N}e(\alpha n)=\int_{1}^{N}e(\alpha t)dt+O(1+N||\alpha||)=I(\alpha)+O(1+N||\alpha||) \end{aligned}$.
	\end{center}	
	Therefore, with $ ||\alpha||\leq L^{\sigma}/(WN) $, we have
	\begin{center}
		$ \widehat{\nu}_{b}(\alpha)-\widehat{1_{[N]}}(\alpha)\ll N||\alpha||+Ne^{-C_{3}\sqrt{L}}\ll Ne^{-C_{3}\sqrt{L}} $.
	\end{center}		
\end{proof}

Finally, let $ \alpha \in \mathfrak{M}(a,q) $ with $ 2\leq q\leq L^{\sigma} ,\ (a,q)=1 $ and let $ \beta =\alpha-a/q \in [-L^{\sigma}/(WN),\ L^{\sigma}/(WN)]. $

Let $ z \in [W] $ with $ (z,W)=1 $, we focus on the estimation of $ S_{q}^{\ast}(a,z) $. Let 
\begin{center}
	$\begin{aligned} S_{q}^{\diamond}(a,z)=\sum\limits_{\substack{r=0 \\ (z+Wr,Wq)=1}}^{q-1}e_{q}\bigg(a\sum\limits_{\ell=1}^{k}\binom{k}{\ell}W^{\ell-1}z^{k-\ell}r^{\ell}\bigg). \end{aligned}$
\end{center}
Clearly,
\begin{equation}\label{equ12}
S_{q}^{\ast}(a,z)=e_{Wq}\big(a(z^{k}-b)\big)S_{q}^{\diamond}(a,z).
\end{equation}
Since $ (z,W)=1 $, $ S_{q}^{\diamond}(a,z) $ has a slightly simpler form 
\begin{center}
	$\begin{aligned} S_{q}^{\diamond}(a,z)=\sum\limits_{\substack{r=0 \\ (z+Wr,q)=1}}^{q-1}e_{q}\bigg(a\sum\limits_{\ell=1}^{k}\binom{k}{\ell}W^{\ell-1}z^{k-\ell}r^{\ell}\bigg). \end{aligned}$
\end{center}
Let $ q=uv $, where $ u $ is $ w$-$smooth $ and $ (v,W)=1 $. Since $ (u,v)=1 $, there exist $ \bar{u},\ \bar{v}\in \mathbb{Z} $ such that $ u\bar{u}+v\bar{v}=1 $. Therefore, we have 
\begin{center}
	$\begin{aligned}
	S_{q}^{\diamond}(a,z)
	&=\sum\limits_{\substack{r=0 \\ (z+Wr,q)=1}}^{q-1}e_{q}\bigg(a\sum\limits_{\ell=1}^{k}\binom{k}{\ell}W^{\ell-1}z^{k-\ell}r^{\ell}\bigg)\\
	&=\sum\limits_{\substack{r_{1}=0 \\ (z+Wr_{1}v,u)=1}}^{u-1}\sum\limits_{\substack{r_{2}=0 \\ (z+Wr_{2}u,v)=1}}^{v-1}e_{uv}\bigg(a\big(u\bar{u}+v\bar{v} \big)\sum\limits_{\ell=1}^{k}\binom{k}{\ell}W^{\ell-1}z^{k-\ell} \times(r_{1}v+r_{2}u)^{\ell} \bigg)\\
	&=\sum\limits_{\substack{r_{1}=0 \\ (z+Wr_{1}v,u)=1}}^{u-1}\sum\limits_{\substack{r_{2}=0 \\ (z+Wr_{2}u,v)=1}}^{v-1}e_{uv}\bigg(a\big(u\bar{u}+v\bar{v} \big)\sum\limits_{\ell=1}^{k}\binom{k}{\ell}W^{\ell-1}z^{k-\ell}\times\big((r_{1}v)^{\ell}+(r_{2}u)^{\ell}\big)\bigg)\\
	\end{aligned}$
\end{center}
\newpage
\begin{center}
$\begin{aligned}
	&=\sum\limits_{\substack{r_{1}=0 \\ (z+Wr_{1}v,u)=1}}^{u-1}e_{u}\bigg(a\bar{v}\sum\limits_{\ell=1}^{k}\binom{k}{\ell}W^{\ell-1}z^{k-\ell}(r_{1}v)^{\ell}\bigg) \sum\limits_{\substack{r_{2}=0 \\ (z+Wr_{2}u,v)=1}}^{v-1}e_{v}\bigg(a\bar{u}\sum\limits_{\ell=1}^{k}\binom{k}{\ell}W^{\ell-1}z^{k-\ell}(r_{2}u)^{\ell}\bigg)\\
	&=\sum\limits_{\substack{r_{1}=0 \\ (z+Wr_{1},u)=1}}^{u-1}e_{u}\bigg(a_{1}\sum\limits_{\ell=1}^{k}\binom{k}{\ell}W^{\ell-1}z^{k-\ell}r_{1}^{\ell}\bigg)\sum\limits_{\substack{r_{2}=0 \\ (z+Wr_{2},v)=1}}^{v-1}e_{v}\bigg(a_{2}\sum\limits_{\ell=1}^{k}\binom{k}{\ell}W^{\ell-1}z^{k-\ell}r_{2}^{\ell}\bigg),
	\end{aligned}$
\end{center}
where $ a_{1}\equiv a\bar{v} \ (\bmod \ u)  $	and $ a_{2}\equiv a\bar{u} \ (\bmod \ v) $.
Therefore, we have
\begin{equation}\label{equ13}
S_{q}^{\diamond}(a,z)=S_{u}^{\diamond}(a_{1},z)S_{v}^{\diamond}(a_{2},z).
\end{equation}	
Similar to the argument of \cite[Section 4]{Chow}, we have
\begin{equation}\label{equ14}
S_{u}^{\diamond}(a_{1},z)=\begin{cases}\begin{aligned}
h\sum\limits_{r_{1}=0}^{u'-1}e_{u}\bigg(a_{1}\sum\limits_{\ell=1}^{k}\binom{k}{\ell}W^{\ell-1}z^{k-\ell}r_{1}^{\ell}\bigg) \quad if \  h\ |\ k, \end{aligned} \\
0 \qquad \qquad \qquad\qquad \quad  \ \ \ \ \ \ \ \ \qquad\qquad otherwise 
\end {cases}
\end{equation}
and 
\begin{center}
	$ S_{v}^{\diamond}(a_{2},z)\ll q^{{1/2}+\epsilon}, $
\end{center}
where $ h=(u,W) , u'=u/h $.

\begin{lemma}\label{lem 5.6}
	Let $ \alpha \in \mathfrak{M}(q) $ and $ 2\leq q\leq L^{\sigma} $.
	Then 
	\begin{center}
		$ \widehat{\nu_{b}}(\alpha)\ll_{\epsilon} w^{\epsilon-1/2}N $
	\end{center}
	for any $ \epsilon>0 $.
\end{lemma}
\begin{proof}
	We split into three cases :(i)$ u\nmid k $, (ii)$ 1\neq q\ |\ k $, (iii)$ u\ |\ k $ and $ q>w $.
	
	The arguments of case (i) and case (iii) are the same as those of Chow in \cite[Section 4]{Chow}. Here we only verify case (ii). During the verification of case (ii), just like Chow, we also make use of the definition of $ W $. If we define $ W $ by $ W:=\prod_{1<p\leq w}p^{k} $, then we would be unable to verify case (ii) for some certain small values of $ k $. 
	
	Trivially $ k\ |\ W $ provided that $ N $ is large enough. Since $ 1\neq q\ |\ k $, we have $ q\ |\ k\ |\ W $. Therefore, $ q=u=h $ and $ v=u'=1 $. By (\ref{equ13}) and (\ref{equ14}), we have $ S_{q}^{\diamond}(a,z)=q $. It follows that, by (\ref{equ12}),
	\begin{center}
		$\begin{aligned} \sum\limits_{\substack{z \in [W]\\ z^{k}\equiv b (\bmod W)}}S_{q}^{\ast}(a,z)=q\sum\limits_{\substack{z \in [W]\\ z^{k}\equiv b (\bmod W)}}e_{Wq}(a(z^{k}-b)).\end{aligned} $
	\end{center} 
	Write
	\begin{center}
		$ z = x+\dfrac{W}{k}y $
	\end{center}
	with $ x \in [W/k] $ and $ y \in \{0,\ldots,k-1\} $. Then
	\begin{center}
		$\begin{aligned} z^{k}= \bigg(x+\dfrac{W}{k}y\bigg)^{k}= x^{k}+\sum\limits_{\ell=1}^{k}\binom{k}{\ell}x^{k-\ell}\bigg(\dfrac{W}{k}y\bigg)^{\ell} .\end{aligned}$
	\end{center}
In view of the definition of $ W $, we have  
	\begin{center}
		 $ z^{k}\equiv x^{k}\ (\bmod \ W) $ ,
	\end{center}
 Therefore, again using  the definition of $ W $, we have
	\begin{equation}\label{equ15}
		\begin{aligned} q^{-1}\sum\limits_{\substack{z \in [W]\\ z^{k}\equiv b (\bmod W)}}S_{q}^{\ast}(a,z)=\sum\limits_{\substack{x \in [W/k]\\ x^{k}\equiv b (\bmod W)}}e_{Wq}(a(x^{k}-b))\sum\limits_{y =0}^{k-1}e_{q}(ax^{k-1}y)  .\end{aligned}
	\end{equation} 
	Since $ (b,W)=1 $ and $ q\ |\ W $, we have $ (x,q)=1 $ when $ x^{k}\equiv b (\bmod\ W) $.
	Therefore, with $ (a,q)=1,\ q>1 $ and $ q\ |\ k $, we have 
	\begin{center}
		$\begin{aligned} \sum\limits_{y=0}^{k-1}e_{q}(ax^{k-1}y)=0. \end{aligned}$
	\end{center}
	Therefore, by (\ref{equ15}),
	\begin{center}
		$\begin{aligned} \sum\limits_{\substack{z \in [W]\\ z^{k}\equiv b (\bmod W)}}S_{q}^{\ast}(a,z)=0.\end{aligned} $
	\end{center}
	By Lemma \ref{lem 5.4}, we have
	\begin{center}
		$ \widehat{\nu_{b}}(\alpha)\ll Ne^{-C_{3}\sqrt{L}}\ll_{\epsilon} w^{\epsilon-1/2}N $.
	\end{center}
\end{proof}
Since $ \alpha \in \mathfrak{M}(a,q) $ with $ 2\leq q\leq L^{\sigma} ,(a,q)=1 $ and $ \beta =\alpha-a/q \in [-L^{\sigma}/(WN),\ L^{\sigma}/(WN)] $, we have
\begin{center}
	$ ||\alpha||=||a/q+\beta||\geq ||a/q||-||\beta||\geq q^{-1}-|\beta|\geq q^{-1}-L^{\sigma}/(WN)\gg L^{-\sigma} $.
\end{center} 
Therefore,
\begin{center}
	$ \widehat{1_{[N]}}(\alpha)\ll ||\alpha||^{-1}\ll L^{\sigma} $.
\end{center}
Coupling this with Lemma \ref{lem 5.6}, we have the following lemma.
\begin{lemma}\label{lem 5.7}
	Let $ \alpha \in \mathfrak{M}(q) $ with $ 2 \leq q \leq L^{\sigma} $ and let $ \epsilon >0 $ . Then
	\begin{center}
		$ \widehat{\nu}_{b}(\alpha)-\widehat{1_{[N]}}(\alpha)\ll_{\epsilon} w^{\epsilon-1/2}N $.
	\end{center}
\end{lemma}

\subsection{Conclusion}
Now we are ready to prove Proposition \ref{pro 5.1} combining the results from the previous subsections.

$ \mathit{Proof \ of \ Proposition \ 5.1.} $ By Lemma \ref{lem 5.3}, Lemma \ref{lem 5.5} and Lemma \ref{lem 5.7}, Proposition \ref{pro 5.1} follows readily. 
\qquad \qquad \qquad \qquad  \qquad \qquad \qquad \qquad \qquad \qquad \qquad \quad \ \ $ \qedsymbol $

\section{Restriction estimate}
In this section, we will establish the restriction estimate. 

Let $ t=k(k+1)/2 $. Fix $ b\in[W] $ with $ b \in Z(W) $. Let $ \phi:[N]\rightarrow \mathbb{R}_{\geq 0} $ with $ \phi \leq \nu_{b} $. Recall that $ L=\log(WN+W)^{1/k} $. Let $ \psi:=L^{-1}\phi $. Define $ \mu(n):[N]\rightarrow \mathbb{R}_{\geq 0} $ by 
 \begin{equation}\label{equ16}
 	 \mu(n):=\begin{cases}
 	\dfrac{1}{\sigma(b)}kx^{k-1} \qquad if \  Wn+b=x^{k}\in \mathbb{N}^{(k)},   \\
 	0  \qquad\qquad\qquad   otherwise. 
 	\end {cases} 
 \end{equation}
 
 Obviously, $ \psi(n)\leq \mu(n) $. Readers can rest assured that the Möbius function does not appear in this manuscript; therefore, $ \mu $ is always defined as above.
 
 First, we establish the following restriction estimate that has an additional $ N^{\epsilon} $ factor.
\begin{lemma}\label{lem 6.1}
	Let $ \epsilon >0 $. Then
	\begin{center}
		$\begin{aligned} \int_{\mathbb{T}}|\widehat{\psi}(\alpha)|^{2t}d\alpha\ll N^{2t-1+\epsilon} \end{aligned}$.
	\end{center}
\end{lemma}
\begin{proof}
	See \cite[proof of Lemma 8.1]{Sal}.
\end{proof}
\begin{lemma}\label{lem 6.2}
	Let $ s> k(k+1) $. Then, there exists $ q\in (s-1, s) $, such that
\begin{center}
	$ \begin{aligned} \int_{\mathbb{T}}|\widehat{\phi}(\alpha)|^{q}d\alpha\ll_{q} N^{q-1} \end{aligned} $.
\end{center}	
\end{lemma}
\begin{proof}
The proof of this lemma is omitted, since with minor changes the proof of \cite[Lemma 5.1]{Chow} also works for our lemma using Lemma \ref{lem 6.1} in place of \cite[Lemma 5.3]{Chow}. Another point of distinction is that Chow \cite[Section 5]{Chow} defines $ \mu(n) $ by 
\begin{center}
	$\begin{aligned} \mu(n)=\dfrac{1}{\sigma(b)}\sum\limits_{\substack{x\in [X]\\ Wn-b=x^{k}}}kx^{k-1} \end{aligned}, $
\end{center} while we define $ \mu(n):[N]\rightarrow \mathbb{R}_{\geq 0} $ by (\ref{equ16}).	
\end{proof}
Taking $ \phi=f_{b} $ and $ \phi=\mathbf{f}_{b} $ respectively, we have the following two results.	
\begin{proposition}\label{pro 6.3}
Let $ s>k(k+1) $. For $ b\in [W] $ with $ b\in Z(W) $, there exists $ q\in (s-1,s) $ such that \begin{center}
	$ \| \widehat{f}_{b} \|_{q} \ll_{q} N^{1-1/q} $.
\end{center}	
\end{proposition}
\begin{proposition}\label{pro 6.4}
	Let $ s>k(k+1) $. For $ b\in [W] $ with $ b\in Z(W) $, there exists $ q\in (s-1,s) $ such that \begin{center}
		$ \| \widehat{\mathbf{f}}_{b} \|_{q} \ll_{q} N^{1-1/q} $.
	\end{center}	
\end{proposition}

\section{Transference Principle and Proof of Theorem 1.1-1.2}
In this section we will use Salmensuu's transference lemma to prove our main theorem. In \cite{Sal}, Salmensuu applied the transference principle to prove a transference lemma (\cite[Proposition 3.9]{Sal}) and used the transference lemma to investigate when a positive density subset of $ k $th powers forms an asymptotic additive basis. The main idea of Salmensuu is to transfer an additive combinatorial result from the integers to a sparse subset of the integers.

First, we introduce some definitions.
\begin{definition}
	Let $ \eta > 0 $ and $ N \in \mathbb{N} $. We say that a function $ f :[N] \rightarrow \mathbb{R}_{\geq 0} $ is $ \eta$-$pseudorandom  $ if there exists a majorant function $ \nu_{f} $ such that $ f\leq \nu_{f} $ pointwise and $ ||\widehat{\nu_{f}}-\widehat{1_{[N]}}||_{\infty} \leq \eta N $.	
\end{definition}
\begin{definition}
	Let $ q > 1,\ N \in \mathbb{N}$ and $ K \geq 1 $. We say that a function $ f:[N] \rightarrow \mathbb{R}_{\geq 0} $ is  q-restricted  with  constant  K  if $ ||\widehat{f}||_{q} \leq KN^{1-1/q} $.
\end{definition}

Salmensuu's transference lemma is the following.
\begin{proposition}(\cite[Proposition 3.9]{Sal})\label{pro 7.3}
	Let $ s\geq 2,\ s-1<q<s,\ K \geq 1 $ and $ \epsilon, \eta \in (0,1). $ Let $ N \in \mathbb{N} $  and,\ for each  i $\in \{1,\ldots,s\}$ let $ f_{i}:[N]\rightarrow \mathbb{R}_{\geq 0} $ be  $ \eta$-$pseudorandom $ and $ q$-$restricted \ with \ constant \ K$. Assume also that 
	\begin{center}
	$ \mathbb{E}_{n \in [N]}f_{1}(n) +\cdots+f_{s}(n)>s(1+\epsilon)/2 $	
	\end{center}
	and 
	\begin{center}
	$ \mathbb{E}_{n \in [N]}f_{i}(n)>\epsilon/2 $
	\end{center}
	for all $ i \in \{1,\ldots,s\}$. Write $ \kappa :=\epsilon/32$. Assume that $ \eta $ is sufficiently small depending on $ \epsilon,\ K,\ q $ and $ s $. Then, for all $ n \in \bigg((1-\kappa^{2})\dfrac{sN}{2},\ (1+\kappa)\dfrac{sN}{2}\bigg) $, 
	\begin{center}
		$ f_{1}\ast\cdots\ast f_{s}(n)\geq c(\epsilon,s)N^{s-1}, $
	\end{center}
	where $ c(\epsilon,s) > 0 $  depends only on $ \epsilon $ and $ s $.
\end{proposition}

However, by Propositions \ref{pro 4.9}-\ref{pro 4.10}, Proposition \ref{pro 5.1} and Propositions \ref{pro 6.3}-\ref{pro 6.4}, we discover that this transference lemma is also valid when we use it to investigate the density version of Waring-Goldbach problem. Recall the definition of $ n_{0} $ at the beginning of Section 3. Our aim is to prove $ n_{0} \in sA^{(k)}  $ (Theorem \ref{thmcont:1.1}) or $ n_{0} \in s\mathcal{P}^{(k)}  $ (Theorem \ref{thmcont:1.2}).

$ \mathit{Proof \ of \ Theorem \ 1.1 \ and \ Theorem \ 1.2 } $. The proof of Theorem \ref{thmcont:1.1} and Theorem \ref{thmcont:1.2} is omitted, since it follows directly by
repeating the arguments in \cite[proof of Theorem 1.1 in subsection 4.3]{Sal}.
  \qquad  \qquad \qquad  \  \   $ \qedsymbol $

\section*{Acknowledgements}
The author would like to thank Professor Yonghui Wang for constant encouragement and for many valuable guidance, and thank Wenying Chen for helpful discussions.

\bibliographystyle{amsplain}

\begin{thebibliography}{10}
	\bibitem {KW} A.V. Kumchev, T.D. Wooley, On the Waring-Goldbach problem for seventh and higher powers,
	Monatsh. Math. 183 (2017) 303–310. \url{https://doi.org/10.1007/s00605-016-0936-7}
	\bibitem {Gre} B. Green, Roth’s theorem in the primes, 
	Ann. Math. 161 (2005) 1609–1636. \url{https://doi.org/10.4007/annals.2005.161.1609}
	\bibitem {LP} H. Li, H. Pan, A density version of Vinogradov’s three primes theorem, Forum Math. 22 (2010) 699–714. \url{https://doi.org/10.1515/forum.2010.039}
	\bibitem {Sal} J. Salmensuu, A density version of Waring’s problem, Acta Arith. 199 (2021) 383–412. \url{https://doi.org/10.4064/aa200601-1-2}
	\bibitem {Hua1} L.K. Hua, Some results in the additive prime-number theory, Q. J. Math. 9 (1938) 68–80. \url{https://doi.org/10.1093/qmath/os-9.1.68}
	\bibitem {Hua2} L.K. Hua, Additive Theory of Prime Numbers, Translations of Mathematical Monographs, vol. 13. American Mathematical Society, Providence, RI, 1965. \url{https://doi.org/10.1090/mmono/013}
	\bibitem {Vau} R.C. Vaughan, The Hardy-Littlewood Method, Second edition. Cambridge Tracts in Mathematics, vol. 125. Cambridge University Press, Cambridge, 1997.  \url{https://doi.org/10.1017/CBO9780511470929}
	\bibitem {Chow} S. Chow, Roth-Waring-Goldbach, Int. Math. Res. Not.  (2018) 2341–2374. \url{https://doi.org/10.1093/imrn/rnw307}
	\bibitem {Shao} X. Shao, A density version of the Vinogradov three primes theorem, Duke Math. J. 163 (2014) 489–512. \url{https://doi.org/10.1215/00127094-2410176}
	
	
	
	
	
	
	
	
	
\end{thebibliography}

\end{document}